\documentclass[10pt]{amsart}
\allowdisplaybreaks

\usepackage{amsfonts}
\usepackage{amsmath}
\usepackage{amssymb}
\usepackage{amsthm}
\usepackage{bbm}
\usepackage{amsbsy}
\usepackage{xcolor}
\usepackage{difftrees}
\usepackage{graphicx} \usepackage{enumerate} \usepackage{multicol}
\usepackage{mathrsfs} \usepackage[all,cmtip]{xy}
\usepackage[color=blue!20,textsize=footnotesize]{todonotes}

\newcounter{dummy} \numberwithin{dummy}{section}
\newtheorem{theorem}[dummy]{Theorem}

\newtheorem{definition}[dummy]{Definition}
\newtheorem{proposition}[dummy]{Proposition}
\theoremstyle{remark}
\newtheorem{remark}[dummy]{Remark}
\newtheorem{example}[dummy]{Example}

\newcommand{\calC}{\mathcal{C}}

\newcommand{\calH}{\mathcal{H}}

\newcommand{\calV}{\mathcal{V}}
\newcommand{\calT}{\mathcal{T}}

\newcommand{\mathfrakHol}{\mathfrak{Hol}}
\newcommand{\mathfrakX}{\mathfrak{X}}
\newcommand{\mathfrakT}{\mathfrak{T}}

\DeclareMathOperator{\Diff}{Diff}

\DeclareMathOperator{\id}{id}

\DeclareMathOperator{\spn}{span}

\DeclareMathOperator{\Hol}{Hol}
\DeclareMathOperator{\hol}{\mathfrak{hol}}

\DeclareMathOperator{\Ad}{Ad}

\DeclareMathOperator{\GL}{GL}
\DeclareMathOperator{\PostLie}{PostLie}
\DeclareMathOperator{\Tree}{Tree}
\DeclareMathOperator{\Lie}{Lie}
\DeclareMathOperator{\gl}{\mathfrak{gl}}

\newcommand{\bbOne}{\mathbbm{1}}

\newcommand{\ve}{\varepsilon}

\numberwithin{equation}{section}

\title[Post-Lie algebra structure of manifolds]{Post-Lie algebra structure of manifolds with constant curvature and torsion}
\author[E.~Grong, H.~Z..~Munthe-Kaas and J.~Stava]{Erlend Grong, Hans Z. Munthe-Kaas and Jonatan Stava}

\address{University of Bergen, Department of Mathematics, P.O.~Box 7803, 5020 Bergen, Norway}
\email{erlend.grong@uib.no}
\email{Hans.Munthe-Kaas@uib.no}
\email{Jonatan.Stava@uib.no}

\subjclass[2020]{53C05, 41A58, 53C30,17D99}

\keywords{Post-Lie algebras, connections, locally homogeneous spaces spaces}

\thanks{The first author is supported by the grant GeoProCo from the Trond Mohn Foundation - Grant TMS2021STG02 (GeoProCo). All authors are supported by the Research Council of Norway through project 302831 Computational Dynamics and Stochastics on Manifolds (CODYSMA)}

\begin{document}

\begin{abstract}
For a general affine connection with parallel torsion and curvature, we show that a post-Lie algebra structure exists on its space of vector fields, generalizing previous results for flat connections. However, for non-flat connections, the vector fields alone are not enough, as the presence of curvature also necessitates that we include endomorphisms corresponding to infinitesimal actions of the holonomy group. We give details on the universal Lie algebra of this post-Lie algebra and give applications for solving differential equations on manifolds.
\end{abstract}

\maketitle

\section{Introduction}
When dealing with numerical solutions of differential equations on manifolds, the major prevalent method in a non-flat space has been through either embeddings or the use of charts. The focus of Geometric Numerical Integration (GNI) algorithms is rather on using global methods, while at the same time conserving the underlying geometric structure. This approach can be used to obtain the solution of discretized versions of both ODEs and PDEs.

One of the main successes of GNI has been its application on Lie groups, see e.g. \cite{LieGMethod,LG1,LG2,LG3}. Lie groups can locally be considered as manifolds equipped with an affine connection $\nabla$ that is flat, i.e., has curvature equal to zero, but possibly with a torsion which is parallel. We can obtain such a connection on a Lie group by defining all left invariant vector fields to be parallel. Conversely, if $\nabla$ is a flat connection with parallel torsion, then all parallel vector fields form a finite dimensional Lie algebra that can be integrated locally to a group structure.

The space of smooth vector fields $\mathfrak{X}_M$ on a manifold $M$ with a connection $\nabla$ can be considered as an algebra over $\mathbb{R}$ called the connection algebra. 
In the case where both the torsion and the curvature are zero, the connection algebra is a pre-Lie algebra. In the case described above, where the curvature is zero while the torsion is parallel, the connection algebra is a post-Lie algebra, see Definition \ref{def:post-Lie}. Post-Lie algebras play an essential part in GNI and many of its methods are based on this algebraic structure. 

In this paper we want to consider a more general type of manifolds that supports an affine connection $\nabla$ which is not necessarily flat, but have parallel curvature and torsion. These can be considered as a local representation of homogeneous spaces. We will show that there exists an underlying post-Lie algebra structure governing vector fields of these spaces as well, as long as we also include endomorphisms corresponding to infinitesimal transformations of the holonomy group.

Pre-Lie algebras is the natural algebraic framework for flat torsion-free connections \cite{CALAQUE2011282}, and post-Lie is the main framework for flat connections with parallel torsion \cite{MK_Wright_08, MKL13}. For a connection $\nabla$ with parallel torsion and curvature there is a subbundle of the frame bundle the sections of which can be considered as a post-Lie algebra which is invariant under the action of a group $H$. This group is isomorphic to the holonomy group of $\nabla$. Rather than using a more algebraic approach such as what is outlined in \cite{al2022algebraic}, we will use the relationship between a manifold and its frame bundle to obtain our realization of the post-Lie algebra structure in an explicit form.

The main result we obtain is the following, which we will prove in Section~\ref{sec:Proof}.

\begin{theorem} \label{th:main}
Let $(M,\nabla)$ be an arbitrary connected affine manifold such that both its curvature $R$ and torsion $T$ are parallel. Let $\Hol_p$ be the holonomy group of $\nabla$ at $p \in M$ with Lie algebra $\hol_p$ and consider the sub-vector bundle $\hol\subseteq T^*M \otimes TM$ of the endomorphisms formed by these Lie algebras.

Let $\mathfrak{X}_M = \Gamma(TM)$ be vector fields on $M$ and $\mathfrak{Hol}_M$ sections of its holonomy bundle. Define binary operations $[ \cdot, \cdot ]$ and $\rhd$ on $\hat{\mathfrak{X}}_M := \mathfrak{X}_M \oplus \mathfrak{Hol}_M$ such that if $x,y \in \mathfrak{X}_M$ and $E,E_1,E_2 \in \mathfrak{Hol}_M$, then
    \begin{align*}
    x \rhd y & = \nabla_x y,  &  [x,y] & =  - T(x,y) + R(x,y), \\
    x \rhd E & = \nabla_x E, &  [E,x] & = -Ex, \\
    E \rhd y & = E y, &  [E_1, E_2]  & = - E_1 E_2 + E_2 E_1. \\
    E_1 \rhd E_2 & = E_1 E_2 - E_2 E_1,
    \end{align*}
Then $(\hat{\mathfrak{X}}_M, [\cdot, \cdot ], \rhd)$ is a post-Lie algebra.
\end{theorem}

\begin{remark}
    The bracket notation $[\cdot,\cdot]$ used for the operation defined in the Theorem is chosen to emphasise the fact that this operation is a Lie bracket that generalize the post-Lie bracket on the vector fields of Lie groups for which we use the same notation, see Section \ref{sec:PostLie}. This bracket should not be confused with the Jacobi bracket on vector fields for which the notation $[\cdot , \cdot]_J$ is used.
\end{remark}

The structure of the paper is as follows. In Section~\ref{sec:Prelim}, we will discuss post-Lie algebras and how they are related to flat connections with parallel torsion. We discuss its universal enveloping algebra and also discuss the effect of having a group action of the post-Lie algebra of vector fields. In Section~\ref{sec:PLFrame} we will discuss how vector fields on a non-flat affine manifold $(M,\nabla)$ can be given a post-Lie algebra structure by lifting them to a manifold of frames $N$ with a flat connection. Finally, Section~\ref{sec:Num} reviews the free post-Lie algebra and consider applications of the post-Lie algebra structure to solutions of ODEs on manifolds.

\section{Flat connections, post-Lie algebras and group action}  \label{sec:Prelim}
\subsection{Post-Lie algebras and flat connections} \label{sec:PostLie}
We start with the definition of the algebraic objects under consideration.
\begin{definition}\label{def:post-Lie}
A post-Lie algebra is a triple $(\mathfrak{A}, [ \cdot, \cdot ], \rhd)$, where $(\mathfrak{A}, [\cdot , \cdot])$ is a Lie algebra and $\rhd: \mathfrak{A} \otimes \mathfrak{A} \to \mathfrak{A}$ is a binary product with associator $a_\rhd$
$$a_\rhd(x,y,z) = x \rhd (y \rhd z) - (x \rhd y) \rhd z,$$
such that for any $x,y,z \in \mathfrak{A}$,
\begin{align*}
x \rhd [y,z] & = [x \rhd y, z] + [y, x \rhd z] , \\
[x,y] \rhd z & = a_{\rhd}(x,y,z) - a_{\rhd}(y,x,z) .
\end{align*}
\end{definition}
\noindent Such structures have many applications in numerics, see e.g. \cite{CEFMK19} for more details. In particular, the free post-Lie algebra is of importance, see Section~\ref{sec:FreePostLie}.

Post-Lie algebras appear in the following way in relation to geometry. Let $(N, \nabla)$ be a smooth manifold endowed with an affine connection, where the connection $\nabla$ has curvature $R$ and torsion $T$. We consider the space of all vector fields $\mathfrak{X}_N = \Gamma(TN)$ on $N$. Write $[x,y]_J = x*y - y*x$ for the Jacobi bracket of vector fields, where $x*y$ is the second order operator
$$(x*y)\phi = x(y\phi).$$
We use the affine connection $\nabla$ to define an operation $x \rhd y = \nabla_x y$ on $\mathfrakX_N$. Assume that the curvature $R =0$ vanishes and that the torsion $T$ is parallel. We define the bracket $[ \cdot , \cdot ]$ by
$$T(x,y) = - [x,y] = x \rhd y - y \rhd x - [ x,  y]_J.$$
The triplet $(\mathfrak{X}_N, [ \cdot, \cdot], \rhd)$ is a post-Lie algebra, see \cite[Section~2]{MKL13}. Note that the bracket presented in Theorem  \ref{th:main} is a generalization of the bracket defined here.

\subsection{Universal enveloping algebra of post-Lie algebras} \label{sec:Universal}
If we have a post-Lie algebra, then we can extend its structure to its universal enveloping algebra. Let $U(\mathfrak{X}_N)$ be the universal enveloping algebra of the Lie algebra $(\mathfrak{X}_N, [ \cdot , \cdot ])$, and denote its group operation by $\cdot$ such that $[x,y] = x\cdot y - y\cdot x$. We can then extend $\rhd$ to $U(\mathfrak{X}_N)$ according to the rules
\begin{equation} \label{ExtendU} \left\{ \begin{split}
\bbOne \rhd x &= x, \\
x \rhd (y\cdot  z)  & = (x \rhd y) \cdot  z + y \cdot  (x \rhd z), \\
(x\cdot y) \rhd z & = x \rhd (y \rhd z) - (x \rhd y) \rhd z,
\end{split} \right.
\end{equation}
making $(U(\mathfrak{X}_N), \cdot , \rhd)$ a D-algebra. We can identify this D-algebra with the space of all differential operator tensors on $N$. Let $\mathfrak{T}_N = \oplus_{j=0}^\infty \mathfrak{X}^{\otimes j}_N$ be the tensor algebra of $\mathfrak{X}_N$, where the tensor product $\otimes$ is over $C^\infty(N)$. Relative to the connection~$\nabla$, introduce its Hessian by
$$\nabla_{x,y}^2 = \nabla_x \nabla_y - \nabla_{\nabla_x y},$$
We can iteratively construct covariant derivatives
$$\nabla_{x_0,x_1, \dots, x_k}^{k+1} =\nabla_{x_0} \nabla_{x_1, \dots, x_k}^{k} - \nabla_{\nabla_{x_0} x_1, \dots, x_k}^{k} - \cdots - \nabla_{x_1, \dots, \nabla_{x_0} x_k}^{k}.$$
For every element $\chi \in \mathfrak{T}_N$, we define $\chi \mapsto D_{\chi}$ as the linear map such that
$$D_{x_1 \otimes \cdots \otimes x_k} = \nabla^k_{x_1, x_2, \dots, x_k},$$
and with $D_{\bbOne} = \id$. We observe that
$$D_{x \otimes y} - D_{y \otimes x } = - D_{T(x,y)} = D_{[x,y]} = : [D_{x}, D_{y}].$$
We will use notation
$$D_{\chi_1} \rhd D_{\chi_1} = D_{D_{\chi_1} \chi_2}, \quad \text{and} \quad D_{\chi_1} \cdot  D_{\chi_2} = D_{\chi_1} D_{\chi_2} - D_{\chi_1} \rhd D_{\chi_2 } = D_{\chi_1 \otimes \chi_2}.$$
Then see that there is an algebra isomorphism between $U(\mathfrak{X}_N)$ and $D_{\mathfrak{T}_N}$ by $x_1 \cdot x_2 \cdots x_k \mapsto D_{x_1} \cdot D_{x_2} \cdots D_{x_k}$. We note that by this definition
$$x \cdot y = x* y - x\rhd y.$$

\subsection{Group actions on Post-Lie algebras}\label{sec: G action on pLa}
We continue with our assumption that $(N, \nabla)$ is a smooth manifold endowed with an affine connection $\nabla$ that is flat and has parallel torsion. By working on the universal cover if necessary, we may assume that~$N$ is simply connected. For any $p \in N$, define a Lie algebra $\mathfrak{g}_p$ as the vector space $\mathfrak{g}_p = T_p N$ with Lie bracket $[u,v]_{\mathfrak{g}_p} = -T(u,v)$, $u,v \in \mathfrak{g}_p$. If $\gamma:[0,1] \to N$ is a smooth curve in~$N$ from $p_0 = \gamma(0)$ and $p_1 = \gamma(1)$, then parallel transport from $\mathfrak{g}_{p_0}$ to $\mathfrak{g}_{p_1}$ is independent of path since the connection is flat, and a Lie algebra isomorphism since the torsion is parallel.

We choose a given base point $o \in N$, and write $\mathfrak{g} = \mathfrak{g}_o$. Let $Y_A$ be the unique parallel vector field with $Y_A|_{o} = A \in \mathfrak{g}$. We observe the relationship
$$[ Y_A|_p, Y_B|_p ]_{\mathfrak{g}_p} = Y_{[A,B]_{\mathfrak{g}}} |_p, \qquad p \in N, \quad A,B\in \mathfrak{g}.$$
If $F \in C^\infty(N, \mathfrak{g})$ is a function with values in the Lie algebra $\mathfrak{g}$, we can define a vector field $Y_F$, by
$$p \mapsto Y_{F(p)}|_p.$$
As all vector fields on $N$ can be described this way, we can identify $\mathfrak{A} = C^\infty(N, \mathfrak{g})$ with $\mathfrak{X}_N = \Gamma(TN)$.
The induced post-Lie structure on $\mathfrak{A}$ is given by
\begin{equation} \label{PostLieFunctions} [F_1, F_2] = [F_1, F_2]_{\mathfrak{g}}, \qquad F_1 \rhd F_2 = Y_{F_1} F_2.
\end{equation}

From the vector fields $Y_A$, we get an action of the Lie algebra $\mathfrak{g}$ on $N$ defined by $\mathfrak{g} \times N \to TN$, $(A,p) \mapsto Y_A|_p$. Assume that for some sub-algebra $\mathfrak{h} \subseteq \mathfrak{g}$, this action can be integrated to a group action. More precisely, let $H$ be a Lie group with Lie algebra $\mathfrak{h}$ and a Lie group homomorphism $\rho: H \to \Diff(N)$ such that $\rho(\exp(tA))$ is the flow of $Y_A$, $A \in \mathfrak{h}$. Note that this is only possible if the vector fields $Y_A$, $A \in \mathfrak{h}$ are complete.
For reasons that will become clear later, we will write this action on the right, i.e. we write
$$\rho(h)(p) = p h^{-1}, \qquad p \in N, \quad h \in H.$$
For any $h \in H$, define the action $\Ad(h)$ on $\mathfrak{X}_N$,
$$(\Ad(h) x)|_p = \rho(h)_* x |_{\rho(h^{-1} )p}, \qquad h \in H, \quad p \in N, \quad x \in \mathfrak{X}_N.$$
If we identify $\mathfrak{X}_N$ with $\mathfrak{A}$, then
$$(\Ad(h) F)(p) = \Ad(h) (F \circ \rho(h^{-1})) (p)= \Ad(h) F(ph), \qquad p\in N,$$
where the symbol $\Ad(h)$ on the right denotes the usual adjoint action of $H$ on $\mathfrak{g}$. We say that $F$ is \emph{equivariant} with respect to $H$ if $\Ad(H)F = F$. Observe that equivariance is preserved under both post-Lie operations, so we obtain a sub-post-Lie algebra $\mathfrak{A}^H$ of equivariant elements.

\begin{remark} \label{re:GenEq}
More generally, if $(\psi, V)$ is a representation of $H$ on a vector space $V$, then we can define an induced representation on smooth functions in the same vector space $C^\infty(N,V)$. This new representation, which we denote by the same symbol, is defined such that if $F \in C^\infty(N,V)$ is a function, then $(\psi(h)F)(p) = \psi(h)F(ph)$. We say that $F$ is \emph{equivariant} if $\psi(h) F = F$ for any $h \in H$.
\end{remark}

\begin{remark}\label{re:ReductivePair}
In what follows, we will have that $(\mathfrak{g}, \mathfrak{h})$ is a reductive pair, that is, assume that we have a decomposition $\mathfrak{g} = \mathfrak{p} \oplus \mathfrak{h}$ which is invariant under the $\Ad$-representation of $H$. Then any function $F \in C^\infty(N, \mathfrak{g})$ can be decomposed uniquely as the sum of two functions, one in $C^\infty(N, \mathfrak{p})$ and one in $C^\infty(N, \mathfrak{h})$, and $F$ is equivariant if and only if both of these parts are equivariant.
\end{remark}

\section{Post-Lie algebra structure and the frame bundle} \label{sec:PLFrame}
\subsection{A frame bundle with left action.}
We will give a short presentation of the structures found in the frame bundle. For more details, see e.g. \cite{Hsu02,GrSo21,Gro22}. Let $M$ be a connected manifold of dimension $m$. Let $\mathbb{R}^m$ be the Euclidean space with the standard metric and with the standard basis $e_1, \dots, e_m$. For any $x \in M$, consider $FM_x$ as the space of all linear isomorphisms $u: \mathbb{R}^m \to T_xM$. Such a map can be identified with a choice of basis $u_1, \dots, u_m$ of $T_xM$ where $u_j = u(e_j)$. We observe that $\GL(m)$ acts transitively on each $FM_x$ through composition on the right, giving us a principal bundle structure $\GL(m) \to FM \stackrel{\pi}{\to} M$. For reasons that will become clear later, we will prefer to use the left action given by
$$q \cdot u = u \circ q^{-1}, \qquad q \in \GL(m).$$
Corresponding to this action, we define vector fields,
\begin{equation} \label{xiA}
\xi_A|_u =\frac{d}{dt} e^{tA} \cdot u  |_{t=0} =\frac{d}{dt} u \circ e^{-tA} |_{t=0}, \qquad A \in \gl(m).
\end{equation}
We then have the following relations for any $A, B \in \gl(m)$, and $q \in \GL(m)$,
$$[ \xi_A, \xi_B ]_J = -\xi_{[A,B]}, \qquad \Ad(q)\xi_{A} = \xi_{\Ad(q) A}.$$
These vector fields span the vector bundle $\calV = \ker \pi_*$. We emphasize that these canonical vector fields related to the group action are the negative of the usual definition of such vector fields. See Remark~\ref{re:Sign} for an explanation of this choice.

Let $\tau$ be any $(k,l)$-tensor on $M$. We then define its \emph{scalarization} $\overline{\tau}:FM \to (\mathbb{R}^m)^{*,\otimes k} \otimes (\mathbb{R}^m)^{\otimes l}$ such that for any
$a_1, \dots, a_k \in \mathbb{R}^m, b_1^*, \dots, b_l^* \in (\mathbb{R}^m)^*$, we have
$$\overline{\tau}(u)(a_1, \dots, a_k, b_1^*, \dots, b_l^*) = \tau|_{\pi(u)}( u(a_1), \dots, u(a_k), u^* b_1^*, \dots, u^* b_l^*).$$
Consider the standard action of $\GL(m)$ on $\mathbb{R}^m$, which induces an action on its tensor products and duals. By construction, using this representation of $\GL(m)$, each function $\overline{\tau}$ is an equivariant function according to the definition in Remark~\ref{re:GenEq}. Conversely, each equivariant function with values in $(\mathbb{R}^m)^{*,\otimes k} \otimes (\mathbb{R}^m)^{\otimes l}$, correspond to a $(k,l)$-tensor.

Let $\nabla$ be a chosen affine connection on $TM$ with torsion $T$ and curvature $R$. For every curve $t \mapsto \gamma(t)$ defined on some interval $(-\ve,\ve)$, there is a $\nabla$-parallel frame $u(t) \in FM_{\gamma(t)}$ uniquely determined by its initial value $u(0) = u$. Write $\dot u(0) = h_u v \in T_u FM$ for the derivative of $u(t)$ along a curve with $\dot \gamma(0) = v$. Let $\mathcal{H}$ be the set all such derivatives, that is $\calH_u = \{ h_u v  \, : \, v \in T_{\pi(u)} M\}$.
Then
$$TFM = \calH \oplus \calV.$$
For any $a \in \mathbb{R}^m$, we can define a global section $H_a$ of $\calH$ given by $H_a |_u = \sum_{j=1}^n a_j h_u u_j$. We then have the following relations for the Jacobi brackets
\begin{equation} \label{xiHbrackets} [\xi_A, H_a ]_J = - H_{Aa}, \qquad [H_a, H_b]_J = - H_{\overline{T}(a,b)} + \xi_{\overline{R}(a,b)}.\end{equation}
Observe from the definition of $H_a$ that the action of $q \in \GL(m)$ is
$$\Ad(q) H_a = H_{q a} .$$

If $x$ is a vector field on $M$, we can define \emph{the horizontal lift} $hx$ of $x$ by $hx|_u = h_u x|_{\pi(u)}$. We then observe the identity
$$hx = H_{\overline{x}}.$$
Furthermore, if $\tau$ is a tensor, then by definition
$$hx \overline{\tau} = H_{\bar{x}} \overline{\tau} = \overline{\nabla_x \tau}, \qquad \xi_A \overline{\tau} = A \cdot \overline{\tau},$$
where $A \cdot \overline{\tau}$ denotes the induced Lie algebra representation of $\gl(m)$ on $(\mathbb{R}^m)^{*,\otimes k} \otimes (\mathbb{R}^m)^{\otimes l}$ from its standard representation on $\mathbb{R}^m$.

\subsection{Post-Lie algebra structure on the frame bundle and the proof of Theorem~\ref{th:main}} \label{sec:Proof}
We define a connection $\tilde \nabla$ on $TFM$ such that all vector fields $H_a$, $\xi_A$, $a \in \mathbb{R}^m$, $A \in \gl(m)$ are parallel. Denote its torsion by $\tilde T$. This connection is obviously flat.
We want to consider the cases where $\tilde \nabla \tilde T = 0$ holds. Observe that for any $a, b, c \in \mathbb{R}^m$,
\begin{align*}
(\tilde \nabla_{H_a} \tilde T)(H_b, H_c) & = H_{H_a \overline{T}(b,c)} - \xi_{H_a \overline{R}(b,c)} \\
& = H_{\overline{\nabla T}(b,c;a)} - \xi_{\overline{\nabla R}(b,c;a)},
\end{align*}
where we have defined
\begin{align*}
& \nabla T(v_1,v_2;v_3)  =  (\nabla_{v_3} T)(v_1, v_2),
\end{align*}
and used a similar definition for $\nabla R$. It follows that $\tilde \nabla_{\calH} \tilde T =0$ if and only if $\nabla$ has parallel torsion and curvature. However, we cannot ensure that $\tilde \nabla_{\calV} \tilde T$ vanishes on the full frame bundle, and so we need to look at a lower dimensional subspace.

For the remainder of the paper, we will assume that $\nabla T = 0$ and that $\nabla R = 0$.
Define a foliation $\tilde M = FM = (\tilde M_u)_{u \in \tilde M}$ where $\tilde M_u$ consist of all elements that can be reached from $u$ by curves in $FM$ tangent to $\calH$. In other words, $\tilde M_u$ consist of all frames that can be obtained from $u$ by parallel transport. We observe that $\overline{T}$ and $\overline{R}$ are constant on each leaf of the foliation by our assumptions.

Choose a base point $o \in M$ and initial frame $u^o$ and define $N := \tilde M_{u^o}$. We observe that for any $u^p \in N_p$, $p \in M$,
$$\Hol_p = \{ u \circ (u^p)^{-1} \, : \, u \in N_p \},$$
is the holonomy group of $\nabla$ at $p \in M$. Since $M$ is connected, all groups $\Hol_p$ coincide up to isomorphism.
Furthermore, if we define $H = \{ u^{-1} \circ u^o \, : \, u \in N_o\} \subseteq \GL(m)$, then $H$ will act transitively on each fiber of $N$. We obtain from this a principal bundle
$$H \to N \to M.$$
If $\mathfrak{h}$ is the Lie algebra of $H$, then
$$TN = \spn \{ H_a |_u, \xi_{A}|_u \, : \, a \in \mathbb{R}^m, A \in \mathfrak{h}, u \in N \},$$
and by Ambrose-Singer theorem \cite{AmSi53},
$$\mathfrak{h} = \{ \overline{R}(u)(a,b) \, : \, u \in N, a,b \in \mathbb{R}^m\}.$$
Since
$$\frac{d}{dt} \overline{T}(u \cdot e^{At}) =0, \quad \text{ and } \quad \frac{d}{dt} \overline{R}(u \cdot e^{At}) =0,$$
for any $A \in \mathfrak{h}$, $u \in N$, we have that, restricted to the frames in $N$,
$$A \overline{T}(a, b) = \overline{T}(A a, b) + \overline{T}(a, Ab), \qquad A \overline{R}(a,b) = \overline{R}(Aa, b) + \overline{R}(a, Ab) + \overline{R}(a,b) A,$$
for any $a,b \in \mathbb{R}^m$. It follows that for the restriction of $\tilde \nabla$ to $N$ satisfies $\tilde \nabla \tilde T =0$. 
We get a corresponding post-Lie algebra $(\mathfrak{X}_N, [\cdot , \cdot ], \rhd)$ with $\mathfrak{X}_N = \Gamma(TN)$.

In summary, if the original connection $\nabla$ on the manifold $M$ has parallel curvature $R$ and torsion $T$, then the induced connection $\tilde \nabla$ on the frame bundle will be flat and with parallel torsion if we restrict to a submanifold $N$ of frames that can be reached by parallel transport from an arbitrarily chosen initial frame. Functions and tensors of $M$ then appear as $H$-equivariant functions on $N$. Using this relationship, we are able to complete the proof of Theorem~\ref{th:main}.

\begin{proof}[Proof of Theorem~\ref{th:main}]
Let $N$ be defined relative to a choice of initial frame $u^o \in FM_o$, $o \in M$ and with connection $\tilde \nabla$. This has parallel vector fields $H_a$, $\xi_A$. We define a Lie algebra
$$\mathfrak{g}= \spn \{ H_a|_{u^o} , \xi_{A}|_{u^o} \, : \, a \in \mathbb{R}^m, A \in \mathfrak{h} \}$$
with Lie bracket given by $- \tilde T$. We can then consider $\mathfrak{h}$ as a subalgebra of $\mathfrak{g}$, and we identify the vector space $\mathfrak{g}$ with $\mathbb{R}^m \oplus \mathfrak{h}$.

We consider the invariant sub-post-Lie algebra $\mathfrak{X}^H_N$ of $\mathfrak{X}_N$.
Using Remark~\ref{re:ReductivePair}, we know that we can write any $w \in \mathfrak{X}^H_N$ as $w = H_{\overline{x}} + \xi_{\overline{E}}$ for some $x \in \mathfrakX_M = \Gamma(TM)$, $E \in \mathfrakHol_M =  \Gamma(\hol)$. 
To give more details, let $q\in H$ and recall that $\Ad(q)\xi_A = \xi_{\Ad(q)A}$ and $\Ad(q)H_a=H_{qa}$. It follows that $(\mathfrak{g},\mathfrak{h})$ is a reductive pair and we can use Remark~\ref{re:ReductivePair}. Hence, if we consider a vector field $\tilde x = Y_F$, $F \in C^{\infty}(N,\mathfrak{g})$, and write it as $F= F_1 \oplus F_2$ with $F_1 \in C^{\infty}(N,\mathbb{R}^m)$ and $F_2 \in C^{\infty}(N,\mathfrak{h})$, then this vector field is equivariant if and only if both $F_1$ and $F_2$ are $\Ad(H)$ equivariant. Next we define a vector field $x \in \mathfrak{X}_M$ by
\[x_p= u(F_1(u))\]
for $p=\pi(u)$. Since $F_1$ is $\Ad(H)$-equivariant we have
\[uq(F_1(uq)) = uq(\Ad(q^{-1})F_1(u)) = uF_1(u),\]
so $x_p$ is well defined. We also note that $F_1(u) = u^{-1} x_p = \bar{x}(u)$. Similarly, we define $E \in \mathfrakHol_M$ by
\[ E_p = uF_2(u)u^{-1}, \qquad p = \pi(u),\]
leading to $F_2 = \bar{E}$.
    
Next, we use equations \eqref{PostLieFunctions} and \eqref{xiHbrackets} to obtain for any $x,y \in \mathfrak{X}_{M}$ and any $E_1, E_2 \in \mathfrakHol_M$,
\begin{align*}
{[\bar{x}, \bar{y}]} &= - \bar{T}(\bar{x}, \bar{y}) + \bar{R}(\bar{x},\bar{y})  = -\overline{T(x,y)} + \overline{R(x,y)}, \\
{[\bar{E}_1, \bar{y}]} & = \bar{E} \bar{y} = \overline{Ey}, \\
{[\bar{E}_1, \bar{E}_2]} & =  - (\bar{E_1} \bar{E_2} - \bar{E_2} \bar{E_1} ) =  - (\overline{E_1 E_2} - \overline{E_2 E_1} ),
\end{align*}
and
\begin{align*}
\bar{x} \rhd \bar{y} & = H_{\bar{x}} \bar{y} = \overline{\nabla_x y},  &\bar{x} \rhd \bar{E}_1 & = H_{\bar{x}} \rhd \bar{E}_1 = \overline{\nabla_x E_1},\\
\bar{E}_1 \rhd \bar{x} & = \xi_{\bar{E}_1} \bar{x} = \overline{E_1x} &
\bar{E}_1 \rhd \bar{E}_2 & = \xi_{\bar{E}_1} \bar{E}_2 = \overline{E_1 E_2} - \overline{E_2 E_1}.
\end{align*}
Since $\mathfrak{X}_N^H$ is a post-Lie algebra it follows from the equations above that $( \mathfrak{X}_M\oplus\mathfrak{Hol}_M, [\cdot, \cdot ], \rhd)$ as defined in Theorem~\ref{th:main} will be an isomorphic post-Lie algebra. \end{proof}

\subsection{Differential operators and universal enveloping algebra}
Consider $\hat{\mathfrak{X}}_M = \mathfrak{X}_M \oplus \mathfrak{Hol}_M$ with its post-Lie algebra structure as in Theorem~\ref{th:main}. Let $N$ be the holonomy bundle formed by all frames that one can find from parallel transport of a given frame. We want to give a description of its universal enveloping algebra similar to what we did for $\mathfrak{X}_N = \Gamma(TN)$ in Section~\ref{sec:Universal}. Recall the definition of the differential operator $D_{w}$, $w \in \mathfrak{X}_N$ from that section.

Recall that by the proof of Theorem~\ref{th:main}, elements in the equivariant sub-post-Lie algebra~$\mathfrakX^H_N$ are sums of elements on the form $H_{\bar x}$ and $\xi_{\bar E}$, with $x$ and $E$ being sections of respectively $TM$ and $\hol$. We define operators $D_x$ and $D_E$ on tensors on $M$, such that if $\tau$ is any tensor, then
$$\overline{D_x \tau} = D_{H_{\bar x}} \bar{\tau}, \qquad \overline{D_{E} \tau} = D_{\xi_{\bar{E}}} \bar{\tau}.$$
In other words, $D_x$ and $D_E$ are operators on tensors uniquely defined by the fact that they satisfy the product rule with respect to the tensor product, and for any section $z$ and $\alpha$ of respectively $TM$ and $T^*M$,
\begin{equation} \label{ActionxE}
D_x z = \nabla_x z, \quad D_x \alpha = \nabla_x \alpha, \quad D_E z = Ez, \quad D_{E} \alpha = E^*\alpha.\end{equation}

If $\hat{\mathfrak{T}}_M = (\oplus_{j=0}^\infty \mathfrak{X}_M^{\otimes j}) \otimes (\oplus_{j=0}^\infty \mathfrak{Hol}_M^{\otimes j})$, we define $D_\eta$ of $\eta \in \hat{\mathfrak{T}}_M$ by extending the following definition by linearity
$$D_{x_1 \otimes \cdots \otimes x_i \otimes E_1 \otimes \cdots \otimes E_j} = D_{E_j} \cdots D_{E_1} \nabla^i_{x_1, \dots,x_i}$$
We also introduce notation $\eta_1 \cdot \eta_2$ such that
$$D_{\eta_1} \cdot D_{\eta_2} := D_{\eta_1} D_{\eta_2} - D_{D_{\eta_1} \eta_2} = D_{\eta_1 \cdot \eta_2}$$
for $\eta_1, \eta_2 \in \hat{\mathfrak{T}}_M$. Writing out this operation for $\chi_1, \chi_2 \in \oplus_{j=0}^\infty \mathfrak{X}_M^{\otimes j}$ and $\epsilon_1, \epsilon_2 \in \oplus_{j=0}^\infty \mathfrak{Hol}_M^{\otimes j}$, $E \in \mathfrak{Hol}_M$, we have
\begin{align*}
    \chi_1 \cdot \chi_2 & = \chi_1 \otimes \chi_2, &  \epsilon_1 \cdot  \epsilon_2 & = \epsilon_1 \otimes \epsilon_2, \\
    \chi_1 \cdot \epsilon_2 & = \chi_1 \otimes \epsilon_2, &  E \cdot \chi_2 & = \chi_2 \otimes E - D_{E} \chi_2.
\end{align*}
It now follows that we can identify $U(\hat{\mathfrakX}_M)$ with $D_{\hat \mathfrakT_M}$. Equivalently, we can identify $U(\hat{\mathfrakX}_M)$ with $(\oplus_{j=0}^\infty (\mathfrak{X}_M \oplus \mathfrak{Hol}_M)^{\otimes j}) /\sim$, where $\sim$ is the equivalence relation
\begin{align*}
    x \otimes y  - y \otimes x  & \sim  -T(x,y) + R(x,y),\\
    E_1 \otimes x - x \otimes E_1 & \sim - E_1x, \\
    E_1 \otimes E_2 - E_2 \otimes E_1 & \sim E_2 E_1- E_1 E_2.
\end{align*}

\begin{remark} \label{re:Sign}
The actions in \eqref{ActionxE} are the reasons for our convention in the definition of $\xi_A$ in \eqref{xiA}. 
\end{remark}

\section{Free post-Lie algebra and numerical solutions} \label{sec:Num}
\subsection{The free post-Lie algebra} \label{sec:FreePostLie}
We again consider $N$ with a flat connection $\tilde \nabla$, which can be considered as a local Lie group. Following description in Section~\ref{sec:PostLie}, we define a corresponding post-Lie algebra $(\mathfrak{X}_N, [\cdot,\cdot], \rhd)$. Consider a subset $\calC \subseteq \mathfrak{X}_N$ and let $\mathfrak{A}(\calC) \subseteq \mathfrak{X}_N$ be the post-Lie algebra generated by this subset. We want to describe a free object $\PostLie(\calC)$ such that we have a post-Lie algebra homomorphism $\PostLie(\calC) \to \mathfrak{A}(\calC)$ regardless of the dimension and geometry of~$(N,\nabla)$.

We give a description of the free post-Lie algebra, following \cite{BV_2007}, \cite{MKL13} or \cite{KMKL_15}. The free post-Lie algebra $\PostLie(\calC)$ generated by a set $\calC$ is the post-Lie algebra that satisfy the universal condition that for any post-Lie algebra $\mathfrak{A}$ and any map $f: \calC \rightarrow \mathfrak{A}$ there exist a unique post-Lie algebra morphism $\phi :\PostLie(\calC) \rightarrow \mathfrak{A}$ such that $\phi(c)=f(c)$ for any $c \in \calC$. The free post-Lie algebra can be described using planar rooted trees with vertices colored by the set $\calC$ and left grafting. 

Let $\Tree(\calC)$ denote the set of planar rooted trees colored by $\calC$. When $\calC = \{ \ab \}$ consist of just one element, then the first elements of $\Tree(\calC)$ graded by the number of vertices is given by  
\[
\Tree(\{ \ab \}) = \Bigg\{ \ab, \aabb, \aaabbb, \aababb, \aaaabbbb, \aaababbb, \aaabbabb, \aabaabbb, \aabababb, \ldots \Bigg\}.
\]
Any tree can be written as $t(c; \tau_1, \ldots , \tau_r)$ where $c \in \calC$ is the root and $\tau_i \in \Tree(\calC)$ are the branches which is connected to the root in the given order. Let $\calT(\calC)$ be the vector space generated by $\Tree(\calC)$ over $\mathbb{R}$. 
The product $\rhd: \Tree(\calC) \times \Tree(\calC) \rightarrow \calT(\calC)$ is defined by 
\[
\begin{split}
    \eta \rhd c &= t(c; \eta), \\
    \eta \rhd t(c; \tau_1, \ldots, \tau_r) &= t(c; \eta, \tau_1, \ldots ,\tau_r) + t(c; \eta \rhd \tau_1 , \ldots ,\tau_r) + \ldots + t(c; \tau_1, \ldots , \eta \rhd \tau_r),
\end{split}
\]
for $c \in \calC$ and $\eta, \tau_1 , \ldots , \tau_r \in \Tree(\calC)$. As an example, we will have
\[
\AabB \rhd \aababb = \aAabBababb + \aaAabBbabb + \aabaAabBbb.
\]
Let $\Lie(\Tree(\calC))$ be the free Lie algebra generated by all planar rooted trees colored by the set $\calC$, introduced in \cite{MKK_03}. In particular
\[
\Lie(\Tree(\{\ab\})) = \spn_{\mathbb{R}} \left\{ \ab, \aabb, \left[\aabb,\ab\right], \aaabbb, \aababb, \left[\left[\aabb, \ab\right], \ab\right], \left[\aaabbb, \ab\right], [\aababb, \ab], \ldots \right\}.
\]
We may extend the product $\rhd$ to $\Lie(\Tree(\calC))$ by linearity and
\[
\begin{split}
    u \rhd [v,w] &= [u\rhd v, w] + [v, u\rhd w], \\
    [u,v] \rhd w &= u \rhd (v \rhd w ) - (u \rhd v) \rhd w - v \rhd (u \rhd w) + (v \rhd u) \rhd w.
\end{split}
\]
Then $\{ \Lie(\Tree(\calC)) , [\cdot , \cdot ] , \rhd \}$ is the free post-Lie algebra generated by $\calC$. A proof of this can be found in \cite{BV_2007} or \cite{MKL13}. 

The enveloping algebra of the free post-Lie algebra is given as the span of ordered forests $OF_{\calC}$ of ordered trees. For the case when $\calC = \{ \ab \}$ then
\[
U(\PostLie(\{\ab \})) = \spn_{\mathbb{R}} \{ OF_{\{\ab\}} \} := \spn_{\mathbb{R}} \{ \mathbb{I} , \ab, \aabb, \ab\ab, \aaabbb, \aababb, \aabb\ab , \ab\aabb, \ab\ab\ab, \dots \}. 
\]
The post-Lie product extends to the enveloping algebra as in Equation (\ref{ExtendU}).
In addition the enveloping algebra is equipped with a concatenation product and the Grossman-Larson product defined by
\[
A * B = A_{(1)} (A_{(2)} \rhd B) ,  \qquad A,B \in OF_{\calC}
\]
where we use Sweedler notation.

\subsection{Application to solutions of ODEs}
The fundamental problem of numerical integration on $M$ is to approximate the exact flow of a vector field $f\in \mathfrak{X}_M$.
We are considering the following initial value problem,
\[
\frac{d}{dt} \Phi_t = f(\Phi_t), \qquad \Phi_0= \phi \in C^\infty(M,\mathbb{R}),
\]
i.e. we are looking for $\Phi_t = \phi \circ e^{tf}$.
Let $\pi: N \subseteq FM \to M$ be a manifold of all frames that can be obtained by parallel transport of some initial frame $u_0$. The flow $\tilde \Phi_t = \Phi_t \circ \pi$ can then be described as the ODE-solution related to the horizontal lift $hf = H_{\bar{f}}$, i.e.,
\begin{equation} \label{LiftDiffEq}
\frac{d}{dt} \tilde \Phi_t = H_{\bar{f}}(\tilde \Phi_t), \qquad \tilde \Phi_0= \phi \circ \pi.
\end{equation}
The manifold $N$ with a flat connection has the local structure of a Lie group and approximations of solutions for \eqref{LiftDiffEq} can be done for instance when using Runge-Kutta-Munthe-Kaas methods. These methods depend on the post-Lie structure of the vector fields on a Lie group. Theorem \ref{th:main} suggests that any numerical method depending on the post-Lie structure could also be used directly on $M$ as long as we modify the algebraic framework. We will give some details about this below.  We provide a description of Lie-Butcher series and explain its relevance for computing flows on a manifold, following \cite{KMKL_15}. Let $\mathfrak{g}_{\calC} = \PostLie(\calC)$ denote the free post-Lie algebra generated by a set $\calC$. The enveloping algebra $U(\mathfrak{g}_{\calC})$ is a vector space with a grading over the number of nodes. Let $U(\mathfrak{g}_{\calC})^*$ be the graded completion of the vector space $U(\mathfrak{g}_{\calC})$, consisting of all infinite sums of its elements with the inverse limit topology. We introduce a paring $\langle \cdot , \cdot \rangle: U(\mathfrak{g}_{\calC})^* \times U(\mathfrak{g}_{\calC}) \rightarrow \mathbb{R}$,
\[
\langle \omega , \omega' \rangle =
    \begin{cases} 0 \quad \text{ if }\omega \neq \omega' \\
    1 \quad \text{ if } \omega = \omega', \\
    \end{cases} \qquad \omega,\omega' \in OF_{\calC},
\]
and extended by linearity. This pairing allows us to identify $U(\mathfrak{g}_{\calC})^*$ with the graded dual.

A universal Lie-Butcher series is an element $\alpha \in U(\mathfrak{g}_{\calC})^*$ and we can write this as an infinite sum
\[
\alpha = \sum_{\omega \in OF_{\calC}} \langle \alpha, \omega \rangle \omega.
\]
The two products in the enveloping algebra give rise to two exponentials
\[
\begin{split}
    \exp^*(\alpha) &= 1 + \alpha + \frac{1}{2!} \alpha*\alpha + \frac{1}{3!} \alpha*\alpha*\alpha + \ldots, \\ 
    \exp^{\cdot}(\alpha) &= 1 + \alpha + \frac{1}{2!} \alpha \cdot \alpha + \frac{1}{3!} \alpha \cdot \alpha \cdot \alpha + \ldots .
\end{split}
\]
We note that
\begin{align*} 
\frac{d}{dt} \exp^*(t\alpha) & =  \exp^*(t\alpha) * \alpha  =\exp^{\cdot}(t\alpha) \cdot (\exp^{\cdot}(t\alpha) \rhd \alpha ), \\
\frac{d}{dt} \exp^{\cdot}(t\alpha) & = \exp^{\cdot}(t\alpha) \cdot \alpha.
\end{align*}
Let $(M,\nabla)$ be a connected affine manifold such that the curvature and torsion are parallel, and let $(\hat{\mathfrak{X}}_M, [\cdot, \cdot], \rhd)$ be the associated post-Lie algebra given in Theorem~\ref{th:main}. For any $f \in \mathfrak{X}_M$ there is a post-Lie homomorphism $F_f : U(\mathfrak{g}_{\{\ab\}}) \rightarrow U(\hat{\mathfrak{X}}_M)$ uniquely determined by $F_f(\ab) = f$. Any universal Lie-Butcher series $\alpha \in U(\mathfrak{g}_{\calC})^*$ together with a chosen $f \in \mathfrak{X}_M$ will then give rise to a Lie-Butcher series over $\hat{\mathfrak{X}}_M$
\[
B_f(\alpha) = \sum_{\omega \in OF_{\{\ab\}}} \langle \alpha , \omega \rangle F_f(\omega).
\]
By defining $f \rhd \phi = f(\phi)$ as the derivative with respect $f$ for any smooth function $C^{\infty}(M,\mathbb{R})$ we get a post-Lie action $\rhd: U(\hat{\mathfrak{X}}_M) \times C^{\infty}(M, \mathbb{R}) \rightarrow C^{\infty}(M,\mathbb{R})$ by extending $\rhd$ according to (\ref{ExtendU}). The exponential with respect to the Grossman-Larson product models the exact flow of a vector field $f \in \mathfrak{X}_M$ in the sense that for any function $\phi \in C^{\infty}(M,\mathbb{R})$
\begin{equation}\label{eq:exact_flow}
    \exp^*(tf)\rhd \phi = \phi\circ e^{tf},
\end{equation}
where $e^{tf} : M \rightarrow M$ is the exact flow of $f$, see \cite[Prop~3.12]{KMKL_15}.  In order to describe the action of $\exp^\cdot(tf)$, we will need the following remarks.

We first look at the special case when~$\nabla$ is a flat connection, as considered in~\cite{KMKL_15}. In this case $\mathfrak{Hol}_M =0$, and so $\hat{\mathfrak{X}}_M = \mathfrak{X}_M$. If $f$ is a vector field and $p \in M$, we remark that there is a unique parallel vector field $f^p$ called \emph{the frozen vector field at $p$} which satisfies
$$f^p|_p = f|_p.$$
We observe that for a for a frozen vector field $(f^p)^{*j} = (f^p)^{\cdot j}$ from the fact that
$$f^p \cdot f^p = f^p * f^p - f^p \rhd f^p = f^p * f^p.$$
Hence, $\exp^*(f^p) = \exp^\cdot(f^p)$ for these vector fields. For a general vector field $f$, if we write
$$\Phi_t(p_0, p) = (\phi \circ e^{tf^{p_0}})(p),\qquad p_0, p \in M,$$
then $(\exp_p^{\cdot}(tf) \rhd \phi)(u) = \Phi_t(p,p)$. We again refer to \cite{KMKL_15} for details.

For a non-flat connection, we cannot define globally parallel vector field in the same way. However, we note that we only need $f \rhd f = \nabla_{f}f = 0$ for $f^{*j} = f^{\cdot j}$ to hold. This allows us to introduce the following definition.

Write $\exp^\nabla$ be the connection exponential of $\nabla$ with $\exp_p^\nabla = \exp^{\nabla}|_{T_pM}$. For a neighborhood $U_{\hat p}$ of $\hat p$ such that $\exp^{\nabla}_{\hat p}$ is invertible with $\exp_{\hat p}^{\nabla,-1}(U_{\hat p})$ convex, define $P_{\hat p,p}:T_{\hat p} M \to T_p M$ by parallel transport along the unique geodesic from $p$ to $\hat p$ in $U$.

\begin{definition}
Let $f \in \mathfrak{X}_M$ be any vector field on $M$. For $\hat p \in M$, define \emph{the frozen vector field} $f^{\hat p}$ of $f$ at $\hat p$ as the vector field on $U_{\hat p}$ by
$$f^{\hat p}|_p = P_{\hat p, p} f|_{\hat p}.$$
\end{definition}
We then have the following result.
\begin{proposition}
If $f \in \mathfrak{X}_M$ and $\phi \in C^\infty(M,\mathbb{R})$, then
$$(\exp^{\cdot}(tf) \rhd \phi)(p) = \phi\left(\exp_p^{\nabla}(t f^{\hat p}|_p) \right) |_{\hat p = p}, \qquad p \in U_{\hat p}.$$
\end{proposition}

\begin{proof}
Let $N$ be a bundle of frames that can be reached from $u^o$ by parallel transport. Let $\tilde \nabla$ be the flat connection on $N$. Write $\pi:N \to M$. We lift $\phi$ to $N$ by $\tilde \phi =\phi \circ \pi$. Since
$$\tilde \nabla_{H_{\bar{f}},\dots, H_{\bar{f}}}^k \tilde \phi = (\nabla^k_{f,\dots,f} \phi) \circ \pi,$$
we have $\exp^\cdot(tH_{\bar{f}}) \rhd \tilde \phi = (\exp^\cdot(t f) \rhd \phi) \circ \pi$. Since $\tilde \nabla$ is flat, we have that
$$(\exp^{\cdot}(tH_{\bar{f}}) \rhd \tilde \phi)(u) = (\tilde \phi \circ e^{tH_{\bar{f}(\hat u)}} )(u) |_{\hat u =u}$$
Remark that if $u(t) =e^{tH_{\bar{f}(\hat u)}}(u)$, then by defininiton $u(t)$ is defined by parallel transport of $u \in N_p$ along the curve
$$\gamma(t) = \exp^\nabla_{p} (tu\bar{f}(\hat u)) = \exp^{\nabla}_p(t u \hat u^{-1} f|_{\hat p}).$$
In particular, this holds for $u = P_{\hat p,p} \hat u$ which will also be in $N$. Inserting such frames into the formula, we have
$$\gamma(t) = \exp^{\nabla}_p(t P_{\hat p, p} f|_{\hat p})= \exp^{\nabla}(t f^{\hat{p}}|_p).$$
It follows that
\begin{align*}
& \phi(\exp^{\nabla}(t f^{\hat{p}}|_p)) |_{\hat p = p}
= (\tilde \phi \circ e^{t H_{\bar{f}(\hat u)}})(P_{\hat p,p} \hat u) |_{\hat p = p} \\
& = (\exp^{\cdot}(t H_{\bar{f}}) \rhd \tilde \phi)(P_{\hat p,p} \hat u) = (\exp^{\cdot}(t f) \rhd \phi)(p)
\end{align*}
The result follows.
\end{proof}

\begin{example}
We consider the sphere $S^m \subseteq \mathbb{R}^{m+1}$ with the subspace metric. We identify $T_p M$ with the subspace of $\mathbb{R}^{m+1}$ orthogonal to $p \in S^m$. Let $\nabla$ be the Levi-Civita connection which has parallel curvature and vanishing torsion. The exponential map and its inverse is given by
\begin{align*}
\exp_{\hat p}^{\nabla}(v) & = \cos(|v|) \hat{p} + \sin(|v|) \frac{v}{|v|}, \\ \\
\exp_{\hat p}^{\nabla,-1}(p) & = \cos^{-1}(\langle \hat p, p \rangle) \frac{p - \langle \hat p, p \rangle \hat p}{|p - \langle \hat p, p \rangle \hat p|} .
\end{align*}
We observe that parallel transport along geodesics is given such that $w \in T_{\hat p} S^m$, $|w| =1$,
\begin{align*}
    P_{\hat p, \exp_{\hat p}^{\nabla}(tw)}(v) & = v - \langle v, w \rangle w + \langle v, w \rangle p_2(t), \\
    p_2(t) & = - \sin(t) \hat p + \cos(t) w.
\end{align*}
Combining it with the formula for the logarithm, we have
\begin{align*}
    P_{\hat p,p}(v) & = v  + \langle v,w \rangle \left( - \sqrt{1-\langle \hat p, p \rangle^2 } \hat p + (\langle \hat p, p \rangle -1)w \right), \\
    w & = \frac{p - \langle \hat p, p \rangle \hat p}{|p - \langle \hat p, p \rangle \hat p|}.
\end{align*}
defined on the open set $U_{\hat p} = \{ p \in S^m \, : \langle p, p \rangle >0 \}$. For more on geometric integration on spheres, see \cite{Hans23}.
\end{example}

\begin{remark}
Let $u^o$ be a chosen frame and let $N$ be all frames obtained by parallel transport of this frame. Assume that the vector fields $H_a$ and $\xi_A$ are complete, so that we can give $N$ a Lie group structure with $u^o$ being the identity. If $\exp_N: T_{u^o} N \to N$ is the corresponding group exponential map, then
$$P_{p,\exp_{p}^{\nabla} w}  = e^{H_{u^{-1} w}}(u) u^{-1} = (u\cdot \exp_N(H_{u^{-1} w}|_{u^o})) u^{-1}, \qquad u \in N_p.$$
\end{remark}

\bibliographystyle{abbrv}

\end{document}